\documentclass{amsart}
\usepackage[utf8]{inputenc}
\usepackage[margin=1 in]{geometry}
\usepackage{amsmath}
\usepackage{amsfonts}
\usepackage{amssymb}
\usepackage{amsthm}
\usepackage{mathrsfs}
\usepackage{tikz-cd}
\usepackage{enumitem}
\usepackage{lmodern}
\usepackage[symbol]{footmisc}
\usepackage{float}
\usepackage[foot]{amsaddr}

\usepackage[maxbibnames=10]{biblatex}
\newtheorem{theorem}{Theorem}

\newtheorem{corollary}[theorem]{Corollary}
\newtheorem{claim}{Claim}

\theoremstyle{definition}
\newtheorem{definition}[theorem]{Definition}
\newtheorem{example}[theorem]{Example}

\theoremstyle{remark}
\newtheorem{remark}{Remark}

\newcommand{\mc}{\mathcal}
\newcommand{\B}{\mc{B}}
\newcommand{\tn}{\textnormal}
\newcommand{\se}{\subseteq}
\newcommand{\ol}{\overline}
\newcommand{\bs}{\backslash}
\newcommand{\Int}{\tn{Int}}
\newcommand{\Ext}{\tn{Ext}}
\newcommand{\D}{\mc{D}}
\newcommand{\es}{\emptyset}

\title{On the compatible sets expansion of the Tutte polynomial}
\author{Laura Pierson}
\address{Harvard University, Cambridge, MA 02138}
\thanks{The author has no relevant financial or non-financial interests to disclose, and data sharing is not applicable to this article.}
\email{lcpierson73@gmail.com}

\begin{document}

\begin{abstract}
    Kochol \cite{Kochol21} gave a new expansion formula for the Tutte polynomial of a matroid using the notion of \emph{compatible sets}, and asked how this expansion relates to the internal-external activities formula. Here, we provide an answer, which is obtained as a special case of a generalized version of the expansion formula to Las Vergnas's trivariate Tutte polynomials of matroid perspectives \cite{LasVerg}. The same generalization to matroid perspectives and bijection with activities have been independently proven by Kochol in \cite{Kochol22} and \cite{Kochol23} in parallel with this work, but using different methods. Kochol proves both results recursively using the contraction-deletion relations, whereas we give a more direct proof of the bijection and use that to deduce the compatible sets expansion formula from Las Vergnas's activities expansion.
    
\end{abstract}

\maketitle

\section{Introduction}

Let $M=(E,\B)$ be a matroid on a finite ground set $E$ with $\B$ the set of bases. Let $r(M)$ be the rank of $M$ and $r_M(X)$ the rank in $M$ of a subset $X\se E$.
Let $M^*$ be the dual matroid and $r^*$ its rank function. For the fundamentals of matroid theory, we point to Oxley \cite{Oxley} or Welsh \cite{Welsh}.

\medskip
Introduced for graphs by Tutte \cite{Tutte} and generalized to matroids by Crapo \cite{Crapo}, the \emph{Tutte polynomial} admits several equivalent descriptions.  Let us recall its definition in terms of \emph{internal-external activities}:

\begin{definition}
Given a matroid $M$ on ground set $E$ with a total ordering $<$ on $E$, and an arbitrary subset $X \se E$, the set of \emph{internally active} elements of $M$ with respect to $X$ is $$\Int_M(X) := \{e\in X: (E\bs X)\cup e\tn{ contains an }M^*\tn{-circuit with $e$ as its minimal element}\},$$ and the set of \emph{externally active} elements is $$\Ext_M(X) := \{e\in E\bs X: X\cup e \tn{ contains an $M$-circuit with $e$ as its minimal element}\}.$$
\end{definition}

\begin{definition}
[Crapo \cite{Crapo}]
\label{defn:Tutte}
The \emph{Tutte polynomial} for a matroid $M=(E,\B)$ is given by $$T_M(x,y) = \sum_{B\in \B} x^{|\Int_M(B)|}y^{|\Ext_M(B)|}.$$
\end{definition}

Kochol \cite{Kochol21} gave a new expansion of Tutte polynomials using the notion of \emph{compatible set}s:

\begin{definition}
[Kochol \cite{Kochol21}]
Let $M$ be a matroid on ground set $E$ with total ordering $<$. Then a set $X\se E$ is \emph{$(M,<)$-compatible} if there is no $M$-circuit $C$ with $X\cap C=\{\min(C)\}$. The set $\D(M,<)$ is $$\D(M,<):=\{X\se E:X\tn{ is }(M^*,<)\tn{-compatible and }E\bs X\tn{ is }(M,<)\tn{-compatible}\}.$$
\end{definition}

\begin{theorem}
[Kochol \cite{Kochol21}]\label{thm:kochol}
The Tutte polynomial $T_M(x,y)$ for a matroid $M=(E,\B)$ is given by $$T_M(x,y) = \sum_{X\in \D(M,<)} x^{r(M/X)}y^{r^*(M|X)}.$$
\end{theorem}

Kochol originally proved this using the contraction-deletion property of the Tutte polynomial \cite{Kochol21}, but asked whether there is a direct proof using the internal-external activities formula. Here, we answer this question in the more general context of Tutte polynomials of \emph{matroid perspectives}:

\begin{definition}
A pair $(M, M')$ of matroids on a common ground set $E$ is a \emph{matroid perspective} (also called a \emph{morphism of matroids}) if every circuit of $M$ is a union of circuits of $M'$.
\end{definition}

\begin{example}
For graphical matroids, one can obtain a matroid perspective by identifying some of the vertices of the graph with each other (see Figure \ref{fig:example_matroids}), and for general realizable matroids, one can obtain a matroid perspective from a linear map.
\end{example}

Las Vergnas introduced the following \emph{trivariate Tutte polynomial} of a matroid perspective in \cite{LasVerg} and further studied its properties in \cite{LasVerg2} and \cite{LasVerg3}.
It admits several different descriptions, just like the Tutte polynomial of a matroid. (See \cite{GioanHB} and \cite{GioanDM} for more on the Tutte polynomial of matroid perspectives and related generalizations of the Tutte polynomial.) We focus here on the description in terms of internal-external activities:

\begin{theorem}
[Las Vergnas \cite{LasVerg}]\label{thm:lasverg}
For a matroid perspective $(M, M')$ on ground set $E$ with total ordering $<$, the trivariate Tutte polynomial is given by $$T_{M, M'}(x,y,z) = \sum_{\substack{B\tn{ independent in }M, \\ \tn{spanning in }M'}} x^{|\Int_{M'}(B)|}y^{|\Ext_M(B)|}z^{r(M) - r(M') - (r_M(B) - r_{M'}(B))}.$$
\end{theorem}

Note that when $M = M'$, one recovers the internal-external activities formula from Definition~\ref{defn:Tutte}.

\begin{remark}
Trivariate Tutte polynomials can also be defined for a general pair of matroids $M$ and $M'$ (see \cite{GLRV}), but in general this polynomial does not have all positive coefficients or a known internal-external activities expansion.
\end{remark}

Kochol's notion of compatible sets can be extended to matroid perspectives as follows:

\begin{definition}
[Kochol \cite{Kochol22}]
Given a matroid perspective $(M, M')$ on ground set $E$ with total ordering $<$, define $$\mc{D}(M, M', <) = \{X\se E: \tn{$X$ is $((M')^*,<)$-compatible and $E\bs X$ is $(M,<)$-compatible}\}.$$
\end{definition}

Note that when $M = M'$, one recovers that $\mc{D}(M,M',<) = \mc{D}(M,<)$.
Our main theorem (independently proven by Kochol in \cite{Kochol23} using the contraction-deletion relations) gives a bijection between internal-external activities of matroid perspectives and compatible subsets.  To state it, we recall the notion of \emph{minimal bases} (see \cite{Oxley}): A total order $<$ on the ground set of a matroid $M$ induces a total lexicographic order on its bases. Denote by $B_{\min}(M)$ the minimal basis of $M$ under this ordering. Then our main result is as follows:

\begin{theorem}[independently proven by Kochol in \cite{Kochol23}]\label{thm:bijection}
Suppose $(M, M')$ is a matroid perspective on ground set $E$ with total ordering $<$. Then there is a bijection $$\{B\tn{ independent in }M,\tn{ spanning in }M'\} \longleftrightarrow \mc{D}(M, M', <),$$ given by $$f(B) := B\bs \Int_{M'}(B) \cup \Ext_M(B),$$ with inverse mapping $$g(X) := X \bs B_{\min}((M|X)^*) \cup B_{\min}(M'/X).$$
\end{theorem}

The reader may consult Example~\ref{ex:bijection_example} illustrating Theorem~\ref{thm:bijection}.

\begin{remark}
It was proven by Las Vergnas in \cite{LasVerg3} (with a missing detail completed by Gioan in \cite{GioanHB}) that for a matroid perspective $(M, M')$ with ground set $E,$ the Boolean intervals $[B\bs\Int_{M'}(B),B\cup\Ext_M(B)]$ partition $2^E$ (where $B$ ranges over sets that are independent in $M$ and spanning in $M'$). The set $B$ is a special representative of its corresponding interval, and the bijection in Theorem \ref{thm:bijection} can be thought of as associating to $B$ and to its interval another special representative $f(B)$ of the same interval. Thus, this bijection provides a translation between the interpretation of these intervals in terms of bases and a corresponding interpretation in terms of compatible sets.
\end{remark}

We obtain as a corollary the following compatible sets expansion of the Tutte polynomial of a matroid perspective, which matches Kochol's formula from \cite{Kochol22} and reduces to Kochol's original formula from \cite{Kochol21} for the case $M=M'$:

\begin{corollary}[Kochol \cite{Kochol22}]\label{cor:expansion}
For a matroid perspective $(M, M')$, the Tutte polynomial can be expressed as $$T_{M,M'}(x,y,z) = \sum_{X\in \mc{D}(M, M', <)} x^{r(M'/X)}y^{r^*(M|X)}z^{r(M) - r(M') - (r_M(X) - r_{M'}(X))}.$$
\end{corollary}

In addition to the special case $M=M',$ one can also consider the special case where $M'=0$ is the uniform matroid of rank 0. In that case, it follows from the convolution formula for the Las Vergnas Tutte polynomial that $T_{M,0}(x,y,z) = T_M(z+1,y)$ (see \cite{LasVerg}). Thus, setting $z=x-1$ gives the following alternate compatible sets expansion of the ordinary Tutte polynomial as a special case of Corollary \ref{cor:expansion}:

\begin{corollary}\label{cor:M'=0}
The ordinary Tutte polynomial has the compatible sets expansion $$T_M(x,y) = \sum_{X\tn{ s.t. }E\bs X \tn{ is }(M,<)\tn{-compatible}} (x-1)^{r(M/X)}y^{r^*(M|X)}.$$
\end{corollary}

This follows from Corollary \ref{cor:expansion} because if $M'=0$, then there are no $(M')^*$ circuits and therefore all sets $X$ are $((M')^*,<)$-compatible, and also $r(M'/X)=0$ for every $X$, so all the powers of $x$ are 0, while the powers of $z=x-1$ are $r(M)-r_M(X)=r(M/X).$ 

The proofs of Theorem \ref{thm:bijection} and Corollary \ref{cor:expansion} will be given in Section \ref{sec:proof}. We finish this section with an example of the bijection in Theorem \ref{thm:bijection}:

\begin{example}\label{ex:bijection_example}
Let $M$ and $M'$ be the two graphical matroids shown in Figure \ref{fig:example_matroids}, so $E=\{1,2,3,4,5\}$, $r(M)=3$, $r(M')=2$, and $M'$ is formed by identifying two vertices of $M$ with each other. The dual matroid $(M')^*$ is also shown, since it is used for computing internal activities.

\begin{figure}[H]
    \centering
    \includegraphics[width=15cm]{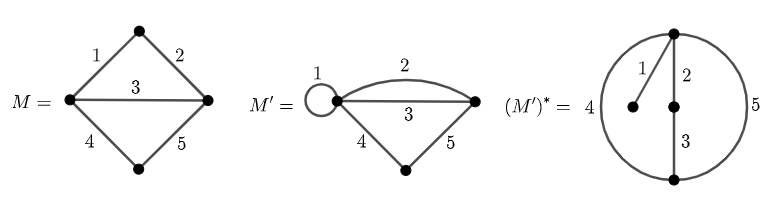}
    \caption{Sample graphical matroid perspective $(M, M')$ used in Example \ref{ex:bijection_example}, along with $(M')^*.$}
    \label{fig:example_matroids}
\end{figure}

Figure \ref{tab:bijection_table} lists the sets $B$ which are independent in $M$ and spanning in $M'$ along with the corresponding internal and external activities, compatible sets $X$, and Tutte polynomial terms.  We find from the table that the Tutte polynomial in this case is given by $$T_{M,M'}(x,y,z) = x^2z+x^2+y^2+xy+2xz+yz+2x+2y+z+1.$$

\begin{table}[h]
    \centering
    \begin{tabular}{c|c|c|c|c}
        $B$ & $\Int_{M'}(B)$ & $\Ext_M(B)$ & $X$ &  Term \\
        \hline
        $\{2,4\}$ & $\{2,4\}$ & $\es$   & $\es$ & $x^2z$ \\
        \hline
        $\{2,5\}$ & $\{2\}$ & $\es$ & $\{5\}$ & $xz$ \\
        \hline 
        $\{3,4\}$ & $\{4\}$ & $\es$ & $\{3\}$ & $xz$ \\
        \hline
        $\{3,5\}$ & $\es$ & $\es$ & $\{3,5\}$ & $z$ \\
        \hline
        $\{4,5\}$ & $\es$ & $\{3\}$ & $\{3,4,5\}$ & $yz$ \\
        \hline
        $\{1,2,4\}$ & $\{2,4\}$ & $\es$ & $\{1\}$ & $x^2$ \\
        \hline
        $\{1,2,5\}$ & $\{2\}$ & $\es$  & $\{1,5\}$ & $x$ \\
        \hline
        $\{1,3,4\}$ & $\{4\}$ & $\es$ & $\{1,3\}$ & $x$ \\
        \hline
        $\{1,3,5\}$ & $\es$ & $\es$ & $\{1,3,5\}$ & 1 \\
        \hline
        $\{1,4,5\}$ & $\es$ & $\{3\}$ & $\{1,3,4,5\}$ & $y$ \\
        \hline
        $\{2,3,4\}$ & $\{4\}$ & $\{1\}$ & $\{1,2,3\}$ & $xy$ \\
        \hline
        $\{2,3,5\}$ & $\es$ & $\{1\}$ & $\{1,2,3,5\}$ & $y$ \\
        \hline
        $\{2,4,5\}$ & $\es$ & $\{1,3\}$ & $\{1,2,3,4,5\}$ & $y^2$ \\
    \end{tabular}
    \vspace{\abovecaptionskip}    \caption{Table illustrating Theorem \ref{thm:bijection} for the matroid perspective in Figure \ref{fig:example_matroids}.}
    \label{tab:bijection_table}
\end{table}

As an example of how to compute one of the rows, consider $B=\{2,3,4\}$. To compute $\Ext_M(B)$, we determine for each $e\not\in B$ whether $e$ is the minimal element of some circuit in $B\cup e.$ If $e=1,$ we get the circuit $\{1,2,3\}$, in which 1 is the minimal element, so 1 is externally active, but if $e=5$, we get the circuit $\{3,4,5\}$, in which 5 is not the minimal element, so 5 is not externally active. Thus, $\Ext_M(B)=\{1\}$. Similarly, for $\Int_{M'}(B)$, we consider the set $E\bs B=\{1,5\}$ in the matroid $(M')^*$, and determine whether each $e\in B$ is the minimal element of some $(M')^*$ circuit in $E\bs B \cup e.$ For $e=2$ and $e=3$, no $(M')^*$-circuits are created by adding them to the set $\{1,5\}$, so they are not internally active, but for $e=4$, the circuit $\{4,5\}$ is created, which does have 4 as its minimal element. Thus, $\Int_{M'}(B)=\{4\}.$ It follows that the corresponding set $X$ is $$X = \{2,3,4\}\bs \{4\} \cup \{1\} = \{1,2,3\},$$ and that the corresponding Tutte polynomial term is $xy$. One can verify that this set $X$ is indeed in $\D(M,M',<)$, because no $(M')^*$-circuit intersects $\{1,2,3\}$ only at the minimal element of the circuit, and no $M$-circuit intersects its complement $\{4,5\}$ only at the minimal element of the circuit. The remaining rows of the table can be checked similarly.
\end{example}

\section{Proof of the bijection and the expansion formula}\label{sec:proof}

We will break the proof of Theorem \ref{thm:bijection} into the following four claims:

\begin{claim}
If $B$ is independent in $M$ and spanning in $M'$, then $X = f(B) \in \D(M, M', <).$
\end{claim}

\begin{claim}
If $X\in \D(M, M', <)$, then $B = g(X)$ is independent in $M$ and spanning in $M'.$
\end{claim}

\begin{claim}
If $B$ is independent in $M$ and spanning in $M'$, then $g(f(B))=B.$
\end{claim}

\begin{claim}
If $X\in \D(M,M',<)$, then $f(g(X))=X.$
\end{claim}

Each claim consists of two parts which are analogous by duality, since $((M')^*, M^*)$ is a matroid perspective if $(M, M')$ is, and since being internally active and being externally active are dual notions.
\setcounter{claim}{0}
\begin{claim}
If $B$ is independent in $M$ and spanning in $M'$, then $X = f(B) \in \D(M, M', <).$
\end{claim}

\begin{proof}
We will show that $E\bs X$ is $(M,<)$-compatible, and the proof that $X$ is $((M')^*,<)$-compatible is analogous by duality. Suppose for contradiction there is an $M$-circuit $C$ such that $$C\cap (E\bs X) = \{\min(C)\}=\{e\}.$$ We will reduce to the case where $C$ contains only elements of $B$, except possibly $e.$ 
If there is some $e'\in C\cap (E\bs B)$ with $e'\ne e$, then $e'\in X\bs B = \Ext_M(B)$, so there is an $M$-circuit $C' \se B \cup e'$ with $e'=\min(C').$ Since $e<e'$, $e\not\in C'$, so by the strong circuit elimination axiom, there is an $M$-circuit $$C''\se (C\cup C')-e'$$ containing $e$. Since all elements of $C'$ besides $e'$ are in $B$, $C''$ contains strictly fewer of $E\bs B$ than $C$ does, and it still has $e$ as its minimal element. 

Thus, we can repeatedly apply the strong circuit elimination axiom in the same manner to eventually get a circuit with $e$ as its minimal element and all other elements in $B$. But this means $e\in \Ext_M(B)$, which is impossible since we assumed $e\not\in X.$ This gives a contradiction, so $E\bs X$ must be $(M,<)$-compatible, and by an analogous argument, $X$ is $((M')^*, <)$-compatible, so $X\in \D(M,M',<).$
\end{proof}

\begin{claim}
If $X\in \D(M, M', <)$, then $B = g(X)$ is independent in $M$ and spanning in $M'.$
\end{claim}

\begin{proof}
Note that $M = M|X \oplus M/X$, so if we take the union of a basis of $M|X$ and a basis of $M/X$, we get a basis of $M$. Now $X\bs B_{\min}((M|X)^*)$ is a basis for $M|X,$ and $B_{\min}(M'/X)$ is an independent set in $M/X$ (since it contains no $(M'/X)$-circuits and therefore no $(M/X)$-circuits, so it is a subset of a basis of $M/X$). It follows that $g(X)$ is a subset of a basis of $M$ and is therefore independent in $M$, and the proof that it is spanning in $M'$ is analogous by duality.
\end{proof}

\begin{claim}\label{claim:fg=id}
If $B$ is independent in $M$ and spanning in $M'$, then $g(f(B))=B.$
\end{claim}

\begin{proof}
Let $X=f(B)$. We wish to show that $$\Ext_M(B)=B_{\min}((M|X)^*), \ \ \ \Int_{M'}(B) = B_{\min}(M'/X).$$ We will prove the first statement, and the second is analogous by duality.

First we claim that $\Ext_M(B)$ is a basis of $(M|X)^*.$ Since $X\bs \Ext_M(B) = B\bs \Int_{M'}(B)$ is a subset of $B$ and is therefore independent in $M|X$, $\Ext_M(B)$ spans $(M|X)^*$. It remains to show that $\Ext_M(B)$ is independent in $(M|X)^*$. If not, then $\Ext_M(B)$ is also not independent in $M^*$, so $E\bs \Ext_M(B)$ (which contains $B$) is not spanning in $M$. Thus, there is no way to extend $B$ to a basis of $M$ except by including an element of $\Ext_M(B)$, which is a contradiction, since for every $e\in \Ext_M(B)$, $B\cup e$ contains an $M$-circuit. It follows that $\Ext_M(B)$ is both independent and spanning in $(M|X)^*$, so it must be a basis.

To see that it is minimal, suppose for contradiction $\Ext_M(B) \ne B_{\min}((M|X)^*)$, and let $$B' = B_{\min}((M|X)^*), \ \ \ e=\min(B'\bs \Ext_M(B)),$$ so $e\in B$ since $e\in X\bs \Ext_M(B)$, and $e$ is smaller than all elements of $\Ext_M(B)\bs B'$ since $B'$ is a smaller basis than $\Ext_M(B)$. Since $X\bs B'$ is a basis of $M|X$, $X\bs B'\cup e$ contains a unique $M$-circuit $C$, which must also contain some $e'\in X\bs B = \Ext_M(B),$ since not all its elements can be in $B$. Thus, there is an $M$-circuit $C'$ in $B\cup e'$ which has $e'$ as its minimal element. Then $e\not\in C'$ because $e<e'$, so by the strong circuit elimination axiom there is an $M$-circuit $$C''\se (C\cup C')-e'$$ containing $e$, which therefore contains strictly fewer elements of $X\bs B$ than $C$ does, since all other elements of $C'$ are in $B$. Applying this repeatedly, we get an $M$-circuit contained entirely in $B$, which is impossible since $B$ is a independent set. It follows that $\Ext_M(B)=B_{\min}((M|X)^*)$, and we can analogously show that $\Int_{M'}(B)=B_{\min}(M'/X)$, thus $g(X)=B.$
\end{proof}

\begin{claim}
If $X\in \D(M,M',<)$, then $f(g(X))=X.$
\end{claim}

\begin{proof}
Let $B=g(X)$, and let $$B'=B_{\min}((M|X)^*), \ \ B''=B_{\min}(M'/X),$$ so $B=X\bs B' \cup B''$. We will show that $B'=\Ext_M(B)$, and the proof that $B''=\Int_{M'}(B)$ is analogous.

First we claim that $B'\se \Ext_M(B)$. If not, let $$e=\min(B'\bs \Ext_M(B)).$$ Then since $e\in B'$ and $X\bs B'$ is a basis for $M|X$, $X\bs B' \cup e$ contains a unique $M$-circuit $C$, and $X\bs B'\se B$, so $C$ is also the only $M$-circuit in $B\cup e.$ Since $e\not\in \Ext_M(B)$, there is some $e'<e$ in $C$. Then $X\bs B'-e'\cup e$ does not contain an $M$-circuit, so it is a basis of $M|X$, which means its complement in $X$, $B'-e\cup e'$, is a basis of $(M|X)^*$, contradicting the minimality of $B'$. It follows that $B'\se \Ext_M(B).$

Second, we claim that $\Ext_M(B)\se B'$. We know $\Ext_M(B)$ is an independent set in $M^*$ by the same argument as in the proof of Claim \ref{claim:fg=id}, so it suffices to show that $\Ext_M(B)\se X.$ If not, suppose $$e\in \Ext_M(B)\bs X,$$ so $e$ is the minimal element of the unique $M$-circuit $C\se B \cup e$. Then all other elements of $C$ are in $B\se X\cup B''$. If they are all in $X,$ then $C$ is an $M$-circuit intersecting $E\bs X$ only at its minimal element, contradiction compatibility. Thus, there is some $e'\in C\cap B''$. By an analogous argument to above, $B''\se \Int_{M'}(B),$ so $e'\in \Int_{M'}(B)$. Then $B-e'\cup e$ is an independent set in $M$ because it does not contain $C$, so its complement $$E\bs(B-e'\cup e) = E\bs B \cup e'-e$$ is spanning in $M^*$. But $B\cup e$ is not independent in $M$, so $E\bs B - e$ is not spanning in $M^*$ and thus must have corank 1, meaning the flat $\ol{E\bs B-e}$ is an $M^*$-hyperplane not containing $e'$. Since $e'$ is internally active in $M'$, it is the minimal element of the unique $(M')^*$-circuit $C'\se E\bs B\cup e'$, so $e\not\in C'$ because $e>e'$. Any $(M')^*$-circuit is a union of $M^*$-circuits since $((M')^*,M^*)$ is a matroid perspective, so there must be some $M^*$-circuit $C''\se E\bs B \cup e'-e$ containing $e'$. But then $$|C''\bs(\ol{E\bs B - e})|=1,$$ which is impossible since $C''$ is a circuit and $\ol{E\bs B-e}$ is a hyperplane. This gives a contradiction, so we must have $\Ext_M(B)\se X$ and therefore $\Ext_M(B)\se B'.$ Thus, $B'=\Ext_M(B),$ and similarly $B''=\Int_{M'}(B)$, so $f(B)=X.$
\end{proof}

Putting these claims together, we conclude that $f$ and $g$ are inverses, so the proof of Theorem \ref{thm:bijection} is complete. We now prove the expansion formula:

\begin{proof}[Proof of Corollary \ref{cor:expansion}]
It follows from Theorem \ref{thm:bijection} that the terms in Las Vergnas's formula (Theorem \ref{thm:lasverg}) are in bijection with the terms of our compatible sets expansion (Corollary \ref{cor:expansion}), and that if $B$ corresponds to $X$ under this bijection, then $$B\bs X = \Int_{M'}(B) = B_{\min}(M'/X), \ \ \ \ X\bs B = \Ext_M(B) = B_{\min}((M|X)^*).$$ It follows from this that $$|\Int_{M'}(B)| = r(M'/X) , \ \ \ |\Ext_M(B)| = r((M|X)^*),$$ which shows that the powers of $x$ and $y$ match in corresponding terms. Thus, it remains to show that the powers of $z$ are also the same, which is equivalent to saying that $$r_M(B)-r_{M'}(B) = r_M(X) - r_{M'}(X).$$ Since $B$ is independent in $M$ and spanning in $M'$, $$r_M(B) = |B|, \ \ \ \ r_{M'}(B) = r(M') \implies r_M(B)-r_{M'}(B) = |B| - r(M').$$ Since $\Int_{M'}(B)$ is a basis for $M'/X$, we get $$r_{M'}(X) = r(M') - r(M'/X) = r(M') - |\Int_{M'}(B)|,$$ and since $\Ext_M(B)$ is a basis for $(M|X)^*$, $$r_M(X) = |X|-r((M|X)^*) = |X|-|\Ext_M(B)| = |B|-|\Int_{M'}(B)|.$$ Subtracting these implies that the powers of $z$ are the same in both cases and completes the proof.
\end{proof}

We can now recover Kochol's original formula (Theorem \ref{thm:kochol}) as the special case $M=M'$, and Corollary \ref{cor:M'=0} as the special case $M'=0.$

\subsection*{Acknowledgements}
The author thanks Christopher Eur for the problem suggestion and for providing helpful comments and references, and the anonymous reviewers for suggesting additional references and making the observation mentioned in Remark 2.

\printbibliography

@article{Kochol21,
  title={Interpretations of the Tutte and characteristic polynomials of matroids},
  author={Martin Kochol},
  journal={Journal of Algebraic Combinatorics},
  volume={53},
  pages={1--9},
  year={2021},
  url={https://link.springer.com/article/10.1007/s10801-019-00914-6}
}

@article{Kochol22,
  title={Interpretations for the Tutte polynomial of morphisms of matroids},
  author={Martin Kochol},
  journal={Discrete Applied Mathematics},
  volume={322},
  pages={210--216},
  year={2022}
}

@unpublished{Kochol23,
  title={One-to-one correspondence between interpretations of the Tutte polynomials},
  year={2023},
  author={Martin Kochol}
}

@article{LasVerg,
  title={On the Tutte Polynomial of a Morphism of Matroids},
  author={Michel Las Vergnas},
  journal={Annals of Discrete Mathematics},
  volume={8},
  issue={Combinatorics 79},
  pages={7--20},
  year={1980}
}

@article{Crapo,
  title={The Tutte polynomial},
  author={Henry H. Crapo},
  journal={Aequationes Mathematicae},
  volume={3},
  year={1969},
  pages={211--229}
}

@article{Tutte,
  title={A contribution to the theory of chromatic polynomials},
  author={William T. Tutte},
  journal={Canadian Journal of Mathematics},
  volume={6},
  year={1954},
  pages={80--91}
}

@book{Oxley,
  title={Matroid Theory},
  author={James Oxley},
  series={Oxford Graduate Texts in Mathematics},
  year={2011},
  doi={10.1093/acprof:oso/9780198566946.001.0001},
  publisher={Oxford University Math},
  location={New York}
}

@book{Welsh,
  title={Matroid theory},
  author={Dominic J. A. Welsh},
  year={1976},
  publisher={Academic Press},
  location={New York}
}

@article{GLRV,
  title={Tutte’s dichromate for signed graphs},
  author={Andrew Goodall and Bart Litjens and Guus Regts and Lluís Vena},
  year={2021},
  month={January},
  journal={Discrete Applied Mathematics},
  volume={289},
  issue={6},
  pages={153--184},
  doi={10.1016/j.dam.2020.09.021}
}

@article{GioanDM,
    author = {Emeric Gioan},
    title = {On Tutte polynomial expansion formulas in perspectives of matroids and oriented matroids},
    journal = {Discrete Mathematics},
    volume = {345},
    year = {2022},
    issue = {112796}
}

@incollection{GioanHB,
    author = {Emeric Gioan},
    title = {The Tutte polynomial of matroid perspectives},
    editor = {Joanna A. Ellis-Monaghan and Iain Moffatt},
    booktitle = {Handbook of the Tutte Polynomial and Related Topics},
    publisher = {CRC Press},
    address = {Boca Raton, FL},
    year = {2022},
    pages = {514--531}
}

@article{LasVerg2,
    author = {Michel Las Vergnas},
    title = {The Tutte polynomial of a morphism of matroids I. Set-pointed matroids and matroid perspectives},
    journal = {Annales de l'institut Fourier},
    address = {Grenoble},
    volume = {49},
    year = {1999},
    pages = {973--1015}
}

@article{LasVerg3, 
    author = {Michel Las Vergnas},
    title = {The Tutte polynomial of a morphism of matroids 5. Derivatives as generating functions of Tutte activities},
    journal = {European Journal of Combinatorics},
    volume = {34},
    issue = {8},
    pages = {1390--1405},
    year = {2013}
}

\end{document}